\documentclass[sn-mathphys]{sn-jnl}% Math and Physical Sciences Reference Style
\jyear{2021}%

\usepackage{amssymb}
\usepackage{amsmath}
\usepackage{amscd}
\usepackage{amsthm}
\usepackage{enumerate}
\usepackage{url}
\usepackage{mathdots}
\usepackage{lscape}

\usepackage{ltablex}
\usepackage{listings}

\theoremstyle{thmstyleone}%
\newtheorem{theorem}{Theorem}%  meant for continuous numbers
\newtheorem{proposition}[theorem]{Proposition}% 
\newtheorem{lemma}[theorem]{Lemma}% 
\newtheorem{conjecture}{Conjecture}

\theoremstyle{thmstyletwo}%
\newtheorem{example}{Example}%

\theoremstyle{thmstylethree}%
\newtheorem{definition}{Definition}%

\newcommand{\IE}{\textit{i.e., }}
\newcommand{\EG}{\textit{e.g., }}

\newcommand{\ZZ}{{\mathbb Z}}
\newcommand{\QQ}{{\mathbb Q}}
\newcommand{\RR}{{\mathbb R}}

\newcommand{\tr}[1]{{\hspace{0.5mm}{}^t\hspace{-0.5mm}#1}}
\newcommand{\abs}[1]{ {\left\lvert #1 \right\rvert} }

\begin{document}

\title[Topological embedding of the space of 3D lattices]{Piecewise-linear embeddings of the space of 3D lattices into $\RR^{13}$ for high-throughput handling of lattice parameters}


\author*[1]{\fnm{Ryoko} \sur{Oishi-Tomiyasu}}\email{tomiyasu@imi.kyushu-u.ac.jp}



\affil*[1]{\orgdiv{Institute of Mathematics for Industry (IMI)}, \orgname{Kyushu University}, \orgaddress{\street{744, Moto'oka}, \city{Nishi-ku}, \postcode{8140015}, \state{Fukuoka}, \country{Japan}}}

%\affil[2]{\orgdiv{Department}, \orgname{Organization}, \orgaddress{\street{Street}, \city{City}, \postcode{10587}, \state{State}, \country{Country}}}

%\affil[3]{\orgdiv{Department}, \orgname{Organization}, \orgaddress{\street{Street}, \city{City}, \postcode{610101}, \state{State}, \country{Country}}}


\abstract{
We present two methods to continuously and piecewise-linearly parametrize rank-3 lattices by vectors of $\RR^{13}$,
which provides an efficient way to judge if two sets of parameters provide nearly identical lattices within their margins of errors.
Such a parametrization can be used to speed up 
scientific computing involving periodic structures in $\RR^3$ such as crystal structures, which includes database querying, detection of duplicate entries, and structure generation via deep learning techniques.
One gives a novel application of Conway's vonorms and conorms, and another is achieved through a natural extension of Ry{\u s}hkov's $C$-type to the setting modulo $3$.
Voronoi vectors modulo 3 obtained in the latter approach provide an algorithm for enumerating of all potential isometries under perturbations of lattice parameters.
}




\maketitle

\section*{Introduction}

Positive-definite symmetric matrices called \textit{Gram matrices} are widely used to parametrize lattices.
In the parametrization, symmetric matrices $S_1$ and $S_2$ of degree $n$ correspond to an identical lattice of rank $n$ if and only if $S_1$ and $S_2$ are equivalent over $\ZZ$, \IE $g S_1 \tr{g} = S_2$ for some $g \in GL_n(\ZZ)$.
Therefore, the parametrization space (\IE \textit{moduli} space of lattices of rank-$n$) is the orbit space ${\mathcal LS}_n := GL_n(\ZZ) \backslash {\mathcal S}_{\succ 0}^n$,
a quotient of the convex cone ${\mathcal S}_{\succ 0}^n$ that consists of all positive-definite symmetric matrices. 

The metric on ${\mathcal LS}_n$ is induced by a metric on ${\mathcal S}_{\succ 0}^n$.
The classical metric on ${\mathcal S}_{\succ 0}^n$, derived from the theory of Riemannian symmetric spaces, is referred to as the \textit{affine-invariant Riemannian metric} in the field of computer vision, which is known to be computationally expensive \cite{Quang2017} as it requires eigenvalue decompositions. Consequently, the metric on ${\mathcal LS}_n$ induced by the affine-invariant metric becomes more computationally demanding, since it involves solving a minimization problem over $g \in GL_n(\mathbb{Z})$.

As is illustrated in the algorithm of Table~\ref{table: calculation of the distance on LS_3},
a similar situation arises even when the \textit{vonorm map} \cite{Conway92} is used, which 
is the continuous parametrization of quadratic forms originally introduced by Voronoi \cite{Voronoi08}; for any $S \in {\mathcal S}_{\succ 0}^n$, 
its \textit{vonorm map} is defined as a map $\Lambda_n := (\ZZ^n/2\ZZ^n) \setminus \{ 2 \ZZ^n \} \rightarrow \RR$ given by
\begin{eqnarray*}
	{\rm vo}_{S} (u + 2\ZZ^n) := \min \{ v S \tr{v} : v \in u + 2\ZZ^n \}.
\end{eqnarray*}
Clearly $S \rightarrow {\rm vo}_{S}$ gives a continuous map ${\mathcal S}_{\succ 0}^n \rightarrow \RR^{\Lambda_n}$, and induces a continuous map $f_n : {\mathcal LS}_n \rightarrow GL_n(\ZZ/2 \ZZ) \backslash \RR^{\Lambda_n}$. 
The comparison algorithm of Table~\ref{table: calculation of the distance on LS_3} works properly in any case as long as $f_n$ is injective, \IE the Conway-Sloane conjecture is true for $n$ \cite{Conway92}.
\begin{conjecture}[Conway-Sloane]
The vororm map ${\rm vo}_{S}(u)$ determines the class of $S \in {\mathcal S}_{\succ 0}^n$ in ${\mathcal{LS}}_n$ uniquely.
\end{conjecture}
The conjecture has been proved true for $n$ up to $5$. 
The cases $n = 4$ and $n = 5$ were proved in~\cite{Vallentin2003} and~\cite{Sikiric2022}, respectively.
The number of matrices $g \in GL_2(\ZZ/2 \ZZ)$ required for the comparison is given by
$\prod_{k=1}^{n} (2^n - 2^{k-1})$,
which evaluates to \(6\), \(168\), \(20{,}160\), and \(9{,}999{,}360\) for $n = 2$--5, respectively.

The other reduction theories also yield various injective maps
$\iota: {\mathcal LS}_n \hookrightarrow \RR^{n (n+1)/2}$,
which are discontinuous at the boundaries of the fundamental domains.  
To address this issue, the use of \textit{nearly reduced bases} (\IE bases that satisfy the inequalities for being reduced within a certain tolerance) has been proposed by \cite{Gruber73} for Buerger reduced cells\cite{Buerger57} used in crystallography.  
Even in the algorithm, the number of nearly reduced bases often exceeds 100 when the rank is 3.

The main result of this article is a method that allows one to tell whether two lattices of rank~3 are nearly identical, simply by comparing their images under an embedding ${\mathcal LS}_3 \rightarrow \RR^{13}$.
This is a natural extension of the recent result for rank-2 lattices in \cite{Kurlin2022}, which makes use of the fact that
vonorm values (or conorms $p_{ij}$ \cite{Conway92}, or root invariants $\sqrt{p_{ij}}$ \cite{Kurlin2022}) sorted in 
ascending order provide a continuous embedding of ${\mathcal LS}_2$ into $\RR^3$. 
Table~\ref{Identification algorithm for lattices of rank 2 that uses vonorm values}
presents a method to compare two rank-2 lattices using a metric $\RR^3$ derived of their vonorm values.

For lattices of rank 3, we propose to use 
a continuous injective map $\iota_s$ or $\iota_m$ from ${\mathcal LS}_3 := GL_3(\ZZ) \backslash {\mathcal S}_{\succ 0}^3$ into $\RR^{13}$.
In what follows, $[ \ ]$ denotes a sorted list in ascending order.
\begin{definition}[$\iota_s$ for Selling reduction]
Let $\iota_s(S) := (f_1(S), f_2(S)) \in \RR^{13}$ be the concatenation of the two vectors $f_1(S) \in \RR^7$, $f_2(S) \in \RR^6$ where
\begin{eqnarray*}
f_1(S)
&:=&
\left[\, {\rm vo}_{S} (g) : g \in \Lambda_3 \right], \\
f_2(S) &:=& \left[\, \sum_{g \in \Lambda_3} \chi (g) {\rm vo}_{S} (g) 
: 
\begin{matrix}
\chi \text{ is a group character } (\ZZ/2\ZZ)^3 \rightarrow \{ \pm 1 \} \\
\text{ that does not satisfy } \\
\chi((1,0,0)) = \chi((0,1,0)) = \chi((0,0,1))  
\end{matrix}
\right].
\end{eqnarray*}
The entries of $f_1(S)$ and $f_2(S)$ are vonorms and conorms in \cite{Conway92}.
In particular, if $S$ is Selling-reduced, 
\begin{eqnarray*}
f_1(S)
&=&
\left[\, v S \tr{v} : v = (1, 0, 0), (0, 1, 0), (0, 0, 1), (1, 1, 0), (1, 0, 1), (0, 1, 1), (1, 1, 1) \right], \\
f_2(S)
&=&
\left[\, -s_{ij} : \text{ non-diagonal entries $s_{ij}$ of $\tilde{S}$ in Eq.(\ref{eq: definition of tilde{S}}) } \right].
\end{eqnarray*}
\end{definition}

\begin{definition}[$\iota_m$ for Minkowski reduction]
For any $n$-by-$n$ metric tensor $S$ and integer $r \ge 2$, define 
the \textit{vonorm map modulo $r$} as the map $(\ZZ^n/r\ZZ^n) \setminus \{ r \ZZ^n \} \rightarrow \RR$ given by
$$
	{\rm vo}_{S, r} (u + r \ZZ^n) := \min \{ v S \tr{v} : v \in u + r \ZZ^n \},
$$
which induces a map $\{ \pm 1 \} \backslash ((\ZZ^n/r \ZZ^n) \setminus \{ 3 \ZZ^n \}) \rightarrow \RR$ due to ${\rm vo}_{S, r} (u + r \ZZ^n) = {\rm vo}_{S, r} (-u + r \ZZ^n)$.
A continuous map from ${\mathcal S}^3_{\succ 0} \rightarrow \RR^{13}$ is given by
\begin{eqnarray*}
	\iota_m(S) := \left[ {\rm vo}_{S, 3} (u + 3\ZZ^3) : \pm u + 3\ZZ^3 \in \{ \pm 1 \} \backslash ((\ZZ^3/3\ZZ^3) \setminus\{ 3 \ZZ^n \}) \right].
\end{eqnarray*}
\end{definition}

It can be easily verified that $\iota_s$ and $\iota_m$ are continuous and piecewise linear, and induce
maps from ${\mathcal LS}_3$ to $\RR^{13}$.
Our main results are the injectivity of $\iota_s$ and $\iota_m$, in addition to a method to calculated $\iota_m$:
\begin{theorem}\label{thm: method to calculate iota_m}
If $S$ is Minkowski-reduced, the $\iota_m(S)$ is the sorted list of the following 13 real values.
\begin{itemize}
\item 
$v S \tr{v}$, where 
$v = (1, 0, 0)$, $(0, 1, 0)$, $(0, 0, 1)$, $(1, 1, 0)$, $(1, 0, 1)$, $(0, 1, 1)$,
$(1, 1, 1)$, $(1,-1, 0)$, $(1, 0, -1)$, $(0, 1, -1)$,
\item $\min \{ v S \tr{v} : v = (-1, 1, 1), (2, 1, 1) \}$, 
\item $\min \{ v S \tr{v} : v = ( 1, 1,-1), (1, 1, 2) \}$,
\item $\min \{ v S \tr{v} : v = (1, -1, 1), (1, 2, 1), (-2, -1, 1), (1, -1, -2) \}$.
\end{itemize}

\end{theorem}

\begin{theorem}\label{thm: injective}
The maps ${\mathcal LS}_3 \rightarrow \RR^{13}$ induced by $\iota_s$ and $\iota_m$ are injective.
\end{theorem}

The proofs of the both theorems rely on an exhaustive search, which was conducted using Magma\cite{MR1484478}. 
In order to prove Theorem~\ref{thm: method to calculate iota_m}, 
an algorithm to enumerate all equivalent primitive $C$-types modulo $r \ge 2$ is provided for the first time. 
For $r = 2$, the method is essentially the same as those briefly given in \cite{Tomiyasu2013a} and \cite{Sikiric2022}.
An enumeration algorithm for modulo 2 was first given in \cite{Ryshkov73}.
$C$-type domains modulo 2 is also called \textit{iso edge domain} in \cite{Sikiric2022}.

As a consequence of Theorem~\ref{thm: injective}, 
any distance in $\RR^{13}$
provides a metric on ${\mathcal LS}_3$.
The inverse $\iota_*^{-1} : \iota_*({\mathcal LS}_3) \rightarrow {\mathcal LS}_3$ of $\iota_s$ and $\iota_m$ can be computed efficiently, owing to the piecewise linearlity, 
by using the inequalities of Selling-reduced or Minkowski-reduced forms.

The strong Whitney embedding theorem ensures that ${\mathcal LS}_n$ is continuously embeddable in the space $\RR^N$ of dimension $N = 2(n(n+1)/2-1)+1 = n (n+1) - 1$.
%If the $N$ is small, less memory and fewer steps are required for the comparison.
The dimension $13$ is larger than the given lower bound $2(6-1)+1 = 11$ for $n = 3$.
In Example~\ref{exam: case of embedding into 12}, we will also see that 
$GL_3(\ZZ/2\ZZ) \backslash \RR^{\Lambda_3}$ can be embedded into $\RR^{12}$ by a polynomial map.
In this case, any perturbations in the image are significantly amplified when its inverse in 
${\mathcal LS}_3$ is calculated.
In view of this, together with the cases in Example~\ref{exam: case of embedding into 7, 8},  
it seems unlikely that a computationally efficient embedding into a lower-dimensional space can be found easily.

If two lattices parameterized by \( S_1 \) and \( S_2 \) are nearly equal, the next step is to find an isometry \( g \in \mathrm{GL}_n(\mathbb{Z}) \) such that \( S_1 \) and \( g S_2 \tr{g} \) are approximately equal. 
For exact parameters, such an isometry can be determined using the algorithm in~\cite{Plesken97}. 
In the presence of unknown perturbations in \( S_1 \) and \( S_2 \), Theorem~\ref{thm: candidate for g} provides a way to construct such a \( g \in \mathrm{GL}_n(\mathbb{Z} \) in the cases \( n = 2 \) and 3. 
In fact, for $n = 2$, 1.~of Theorem~\ref{thm: candidate for g} ensures the following under a general assumption on the error;
$$
	\tr{g} {\mathbf e}_1, 
	\tr{g} {\mathbf e}_2 \in \{ \pm {\mathbf e}_1, 
					\pm {\mathbf e}_2, 
					\pm ({\mathbf e}_1 + {\mathbf e}_2) \}. 
$$
A similar statement can also be proven for rank-3 lattices.

\begin{table}[htbp]
\caption{ 
Algorithm for computing a distance on $\mathcal{LS}_n$ 
using the vonorm maps and any metric $d$ on $\RR^{\Lambda_n}$. 
This provides a proper metric if the Conway-Sloane conjecture holds for $n$.
}
\label{table: calculation of the distance on LS_3}
\begin{tabular}{lp{95mm}}
\hline \\
\multicolumn{2}{l}{{\bf Input}: } \\ 
$S_1, S_2 \in {\mathcal S}_{\succ 0}^n$ & 
: pair of Gram matrices. \\
$d$ & : metric on $\RR^{\Lambda_n}$ ($=\RR^{2^n-1}$) \\
\multicolumn{2}{l}{{\bf Output}: } \\
& distance between the classes of $S_1$ and $S_2$ in ${\mathcal LS}_n$. \\
\multicolumn{2}{l}{{\bf Algorithm}:} \\ 
1: & Calculate ${\rm vo}_{S_i} \in \RR^{\Lambda_n}$ ($i = 1, 2$) \EG using Fincke-Pohst algorithm \cite{Fincke83}. \\
2: & For each $g \in GL_n(\ZZ/2 \ZZ)$, compute $d({\rm vo}_{S_1}, g \cdot {\rm vo}_{S_2})$. \\
3: & Return the minimum of the distances. \\
\hline
\end{tabular}
\end{table}

\begin{table}[htbp]
\caption{
Algorithm for computing a distance on $\mathcal{LS}_n$ 
using any metric $d$ on $\RR^3$}
\label{Identification algorithm for lattices of rank 2 that uses vonorm values}
\begin{tabular}{lp{93mm}}
\hline \\
\multicolumn{2}{l}{{\bf Input}:} \\ 
$S_1, S_2 \in {\mathcal S}_{\succ 0}^2$ & : pair of reduced matrices, 
where $S = (s_{ij}) \in {\mathcal S}_{\succ 0}^2$ is {\it reduced} if
$0 \le -2 s_{12} \le s_{11} \le s_{22}$. \\
$d$ & : metric on $\RR^3$ \\
%$\epsilon > 0$ & a threshold. \\
\end{tabular}
\begin{tabular}{lp{110mm}}
\multicolumn{2}{l}{{\bf Output}:} \\
& distance between the classes of $S_1, S_2$ in ${\mathcal LS}_2$ induced by $d$. \\
\multicolumn{2}{l}{{\bf Algorithm}:} \\ 
1: & For each $S_i$ ($i = 1, 2$), let $\ell_i$ be the vector 
$(
	\tr{{\mathbf e}_1} S_i {\mathbf e}_1,
	\tr{{\mathbf e}_2} S_i {\mathbf e}_2,
	\tr{({\mathbf e}_1+{\mathbf e}_2)} S_i ({\mathbf e}_1+{\mathbf e}_2)
)$ in $\RR^3$, where ${\mathbf e}_1 := (1, 0)$, ${\mathbf e}_2 := (0, 1)$. The obtained $\ell_i$ is a sorted list as it is.
\\
2: & Return $d(\ell_1, \ell_2)$. \\
\\
\hline
\end{tabular}
\end{table}

%\begin{table}[htbp]
%\caption{Algorithm to judge if two lattices of rank 3 are nearly equal using an embedding $\iota_*$ }
%\label{table: identification algorithm}
%\begin{minipage}{\textwidth}
%\begin{tabular}{lp{93mm}}
%\hline \\
%\multicolumn{2}{l}{{\bf Input}:} \\ 
%$S_1, S_2 \in {\mathcal S}_{\succ 0}^3$: & 
%a pair  of Gram matrices.
%\end{tabular}
%\begin{tabular}{lp{110mm}}
%\multicolumn{2}{l}{{\bf Output}:} \\
%& True if the lattice parametrized by $S_1, S_2$ are identical within their margins of error, otherwise return False. 
%\\
%\multicolumn{2}{l}{{\bf Algorithm}:} \\ 
%1: & Using a reduction algorithm, transform $S_1, S_2$ into their Selling-reduced forms if $\iota_* = \iota_s$, or Minkowski-reduced forms if $\iota_* = \iota_m$. \\
%2: & Calculate $\ell_1 := \iota_*(S_1)$ and $\ell_2 := \iota_*(S_2)$. \\
%3: & Return True if $\ell_1$ and $\ell_2$ are nearly equal (\EG ${\rm dist}(\ell_1, \ell_2) < \epsilon$ for a given distance function on the 13-dimensional parameter space and a lower bound $\epsilon > 0$). 
%Otherwise, return False. \\
%\hline
%\end{tabular}
%\end{minipage}
%\end{table}

\section{Notation and background}

Let \( \mathbf{e}_1, \ldots, \mathbf{e}_n \) be the standard basis of \( \mathbb{R}^n \), \IE each \( \mathbf{e}_i \) is the vector with 1 in the \( i \)-th entry and 0 in all other entries.
All the \( n \times n \) symmetric matrices form a linear space of dimension \( n(n+1)/2 \) over $\RR$, which is denoted by \( \mathcal{S}^n \).  
The cone consisting of all the positive-definite matrices in \( \mathcal{S}^n \) is denoted by \( \mathcal{S}^n_{\succ 0} \).

For any linearly independent \( l_1, \ldots, l_n \in \RR^N \),  
\( L = \mathbb{Z} l_1 + \cdots + \mathbb{Z} l_n \subset \mathbb{R}^N \)
is a \textit{lattice} of \textit{rank} $n$, and $l_1, \ldots, l_n$ are a \textit{basis} of $L$.
If \( L \) is full-rank (\IE \( n = N \)), the rank \( n \) is also referred to as the \textit{dimension} of \( L \).
The \textit{Gram matrix} of \( l_1, \ldots, l_n \) is defined as the symmetric matrix  
\( (l_i \cdot l_j)_{1 \le i, j \le n} \),  
where \( u \cdot v \) denotes the Euclidean inner product. 
Any elements of $\mathcal{S}^n_{\succ 0}$ can be regarded as the Gram matrix of a lattice basis.

A Gram matrix \( S = (s_{ij}) \in \mathcal{S}_{\succ 0}^2 \) is \textit{reduced} if the entries satisfy
\[
0 \le -2s_{12} \le s_{11} \le s_{22}.
\]
\( S = (s_{ij}) \in \mathcal{S}_{\succ 0}^3 \) is \textit{Minkowski reduced} if all of the following inequalities hold:
\begin{eqnarray} \label{eq: inequalities for being Minkowski reduced}
s_{11} \le s_{22} \le s_{33}, \quad 
0 \le -2s_{12} \le s_{11}, \quad 
2\abs{s_{13}} \le s_{11}, \nonumber \\
0 \le -2s_{23} \le s_{22}, \quad 
-2(s_{12} + s_{13} + s_{23}) \le s_{11} + s_{22}.
\end{eqnarray}

Let \( \mathbf{e}_1, \ldots, \mathbf{e}_n \) be the standard basis vectors of \( \mathbb{R}^n \), and define \( \mathbf{e}_{n+1} = -\sum_{i=1}^n \mathbf{e}_i \).  
For \( n = 2, 3 \), \( S \in \mathcal{S}_{\succ 0}^n \) is called \textit{Selling reduced} if all non-diagonal entries $s_{ij}$ ($i \ne j$) of the following \( \tilde{S} = (s_{ij}) \in \mathcal{S}_{\succ 0}^{n+1} \) are not positive \cite{Conway97}:
\begin{eqnarray} \label{eq: definition of tilde{S}}
\tilde{S} := \tr{w} S w, \quad 
w := 
\begin{pmatrix}
\mathbf{e}_1 & \cdots & \mathbf{e}_n & \mathbf{e}_{n+1}
\end{pmatrix} \in \mathbb{R}^{n \times (n+1)}.
\end{eqnarray}

For any \( S \in {\mathcal S}^n_{\succ 0} \),
its \textit{vonorm map} is defined as a map $(\ZZ^n / 2\ZZ^n) \setminus \{ 2 \ZZ^n \} \rightarrow \RR_{> 0}$ given by:
\begin{eqnarray*}
\mathrm{vo}_S(v + 2\ZZ^n) := \min \{ w S \tr{w} : w \in v + 2\ZZ^n \}.
\end{eqnarray*}

The corresponding non-zero vectors \( v \in L \) that are the shortest in $l + 2 L$ are called a \textit{Voronoi vector}. This condition can be also written as:
\[
	v \cdot l \le l \cdot l \quad \text{for all } l \in L.
\]
The vonorm map defined in \cite{Conway92} is derived of the Voronoi vectors, which was defined in Voronoi's second reduction theory \cite{Voronoi08}.

The \textit{conorm map} of \( S \) is defined as the Fourier transform of the vonorm map; for any character \( \chi: \ZZ^n / 2\ZZ^n \to \{ \pm 1 \} \), it is defined as
\begin{eqnarray*}
\mathrm{co}_S(\chi) := -\frac{1}{2^{n-1}} \sum_{v + 2\ZZ^n \in \ZZ^n / 2\ZZ^n} \mathrm{vo}_S(v + 2\ZZ^n) \, \chi(v).
\end{eqnarray*}
In the case of $n = 3$, 
the \textit{conorm} values of \( S \) 
coincides with the negatives \( -s_{ij} \ge 0 \) of the non-diagonal entries of $\tilde{S}$.

Using the notation of \cite{Vallentin2003}, we recall the definition of $L$-type domains in \cite{Voronoi08}, also known as \textit{secondary cones}.
The \textit{Dirichlet-Voronoi polytope} of $S$ is defined by: 
\begin{eqnarray*}
	{\rm DV}(S) := \{ x \in \RR^n : x S \tr{x} \leq (x + v) S \tr{(x + v)} \text{ for any } v \in \ZZ^n \}.
\end{eqnarray*}

From the definition, ${\rm DV}(S)$ is the intersection of half-spaces:
\begin{eqnarray*}
	{\rm DV}(S) &=& \bigcap_{0 \neq v \in \ZZ^n}
				\{ x \in \RR^n : x S \tr{x} \leq (x + v) S \tr{(x + v)} \}.
\end{eqnarray*}

As proved in \cite{Voronoi08} and \cite{Conway92}, 
$v \in \ZZ^n$ is a Voronoi vector
if and only if the hyperplane 
$H_{S, v} := \left\{ x \in \RR^n : x S \tr{x} = (x + v) S \tr{(x + v)} \right\}$
intersects ${\rm DV}(S)$.
A tiling of $\RR^n$ is given by the Dirichlet-Voronoi polytopes:
\begin{eqnarray}\label{eq:tiling by Diriclet-Voronoi polytopes}
	\RR^n = \bigcup_{l \in \ZZ^n} ({\rm DV}(S) + l).
\end{eqnarray}

The Delone subdivision is the dual tiling of (\ref{eq:tiling by Diriclet-Voronoi polytopes}).
Let $P_S$ denote the set of extreme points of ${\rm DV}(S)$.
For each $p \in P_S$, define $\Psi_p$ as the set of all Voronoi vectors $v$ satisfying $p \in H_{S, v}$.
Then, every $v \in \Psi_p$ satisfies the relation
$
p S \tr{p} = (p + v) S \tr{(p + v)}.
$
Therefore, the ellipsoid
$
\left\{ x \in \mathbb{R}^N : (x + p) S \tr{(x + p)} = p S \tr{p} \right\}
$
passes through all elements of $\{ 0 \} \cup \Psi_p \subset \mathbb{Z}^n$.
If $L_p$ is the convex hull of $\{ 0 \} \cup \Psi_p$, then the Delone subdivision of $\RR^n$ 
is given by:
\begin{eqnarray}
	{\rm Del}(S) := \bigcup_{ p \in P_S, l \in \ZZ^n } L_p + l.
\end{eqnarray}
The \textit{$L$-type domain} containing $S$ is defined by
$\{ S_2 \in {\mathcal S}^n : {\rm Del}(S) = {\rm Del}(S_2) \}$.

Two domains $D_1, D_2 \subset {\mathcal S}^n_{\succ 0}$ are said to be \textit{equivalent}
if $D_1[g] := \{ g S \tr{g} : S \in D_1 \} = D_2$ holds for some $g \in GL_n(\ZZ)$.
For a fixed $n$, 
the number of the secondary cones is finite up to the action of $GL_n(\ZZ)$ \cite{Voronoi08}.

\section{Generalization of Ry{\v s}kov's $C$-types to the setting modulo $r$}

We recall the definition of Ry{\v s}kov's \( C \)-types in \cite{Ryshkov76}, which was introduced to solve the covering problem in dimension 5. The definition will be extended to consider the Voronoi map modulo $r$.

\begin{definition}
Fix $S \in {\mathcal S}^n_{\succ 0}$ and any integer $r \geq 2$. Define
the \textit{vonorm map modulo $r$} of $S$ as the map $(\ZZ^r / r \ZZ^n) \setminus \{ r \ZZ^n \} \rightarrow \RR_{> 0}$ given by 
\begin{eqnarray}
{\rm vo}_{S, r}(v + r \ZZ^n) := \min \{ w S \tr{w} : w \in v + r \ZZ^n \}.
\end{eqnarray}
Define the set $\Phi_{S, r} \subset \ZZ^n$ by
$$
\Phi_{S, r} := \{ v \in \ZZ^n : v S \tr{v} = {\rm vo}_S(v + r \ZZ^n) \}.
$$
A \textit{$C$-type domain (modulo $r$)} containing $S$ is defined by
$$
{\mathcal D}_r(\Phi_{S, r}) := \{ S_2 \in {\mathcal S}_{\succ 0}^n : \Phi_{S, r} \subset \Phi_{S_2, r} \}.
$$
More generally, for any subset $\Phi \subset \ZZ^n$, its $C$-type domain is defined by 
\begin{eqnarray}
    {\mathcal D}_r(\Phi) := \{ S \in {\mathcal S}^n_{\succ 0} : v S \tr{v} = {\rm vo}_S(v + r \ZZ^n) \text{ for any } v \in \Phi \}.
\end{eqnarray}

\end{definition}

From the definition, 
$v \in \ZZ^n$ is contained in $\Phi_{S, r}$ if and only if
\begin{eqnarray}\label{eq: inequality r/2}
	v S \tr{l} \le (r/2) l S \tr{l} \text{ for any $l \in \ZZ^n$.}
\end{eqnarray}
Therefore, $\Phi_{S, r_1} \subset \Phi_{S, r_2}$ for any integers $2 \le r_1 \le r_2$.
We can also consider lattice vectors $l \in L$ that is the shortest in the class $l + r L$. We call such an $l$ a \textit{Voronoi vector modulo $r$}.

For a fixed $r \ge 2$, $S \in {\mathcal S}_{\succ 0}^n$ is in a \textit{general position} 
if any $u, v \in \Phi_{S, r}$ such that $u \ne \pm v$ belong to distinct cosets in $\ZZ^n/r \ZZ^n$.
In such a case, the number of elements of $\Phi_{S, r}$ is exactly given by
$$
	\abs{ \Phi_{S, r} }
	= 
	\begin{cases}
		r^n + 2^n - 2 & \text{if $r$ is even}, \\
		r^n - 1 & \text{if $r$ is odd}. 
	\end{cases}
$$
For even $r$, $2^n - 1$ is added for the cosets $u + \ZZ^n \in (\ZZ^n/r \ZZ^n) \setminus \{ r \ZZ^n \}$ such that $-u + r \ZZ^n = u + r \ZZ^n$.
If $S$ is in general position, 
${\mathcal D}_r(\Phi_{S, r})$ contains an open neighborhood of $S$.
Otherwise, there exist $u, v \in \Phi_{S, r}$ such that $u \ne \pm v$, $u + r \ZZ^n = v + r \ZZ^n$.
$S$ belongs to the hyperplane:
\begin{eqnarray}
	H( u, v ) := \{ S \in {\mathcal S}^n : u S \tr{u} = v S \tr{v} \}.
\end{eqnarray}

Even for such an $S$, by perturbing the entries of $S$, 
$\tilde{S}$ in general position such that ${\mathcal D}_r(\Phi_{S, r}) \subset {\mathcal D}_r(\Phi_{\tilde{S}, r})$ 
is obtained.
As a result, the following partitioning of ${\mathcal S}^n_{\succ 0}$ is obtained for each $r$.
\begin{eqnarray}\label{eq:decomposition of S_{> 0}^n}
	{\mathcal S}^n_{\succ 0} &=& \bigcup_{ S \in {\mathcal S}_{\succ 0}^n \text{ in general position}  } {\mathcal D}_r(\Phi_{S, r}).
\end{eqnarray}

\begin{example}\label{exam: S in general position}
It is straightforward to verify that if $r$ is odd, the identity matrix $I_n$ satisfies
$$
	\Phi_{I_n, r} = \{ (x_1, \ldots, x_n) \in \ZZ^n : \abs{ x_i } \le r/2 \}.
$$
Consequently, $I_n$ is in general position for odd $r$.
If $r \ge 2$ is even, $T \in {\mathcal S}^n$ is a direction
such that $I_n + \epsilon T$ is in general position
for sufficiently small $\epsilon > 0$, if $S_0 \bullet T := {\rm trace}(S_0 T) \ne 0$ for all $S_0$
expressed as follows:
$$
	S_0 = \tr{u} u - \tr{v} v, \quad u, v \in \Phi_{I_n, r}, \quad u \ne \pm v, \quad u - v \in r \ZZ^n.
$$
Since this excludes only $T$ in a union of finitely many hyperplanes, 
$\tilde{S}$ in general position can be obtained efficiently for any $r \ge 2$.
\end{example}

A $C$-type domain ${\mathcal D}_r(\Phi_{S, r}) \ne \emptyset$ is \textit{primitive} (or \textit{generic} \cite{Sikiric2025}),
if it is provided by some $S$ in general position, and thus contains an interior point.
In general, ${\mathcal D}_r(\Phi)$ is the intersection of the following half-spaces.
\begin{eqnarray*}
	{\mathcal D}_r(\Phi) &=& \bigcap_{ u \in \Phi } \bigcap_{ v \in u + r \ZZ^n} H^{\leq 0}(u, v), \\
	H^{\leq 0}(u, v) &:=& \{ S \in {\mathcal S}^n_{\succ 0} : u S \tr{u} \leq v S \tr{v} \}.
\end{eqnarray*}
As a result of Proposition~\ref{prop:facet and parallelogram law} below,
primitive ${\mathcal D}_r(\Phi)$ is a polyhedral cone for any $r \ge 2$.

%We say that $C$-type domains ${\mathcal D}_1$, ${\mathcal D}_2$ modulo $r$ are  \textit{equivalent}, if ${\mathcal D}_1[g] = {\mathcal D}_2$ for some $g \in GL_n(\ZZ)$.
The following Example~\ref{exam: mod 2} is a direct consequence of Selling reduction for $n = 2$ and 3 \cite{Conway92}. 
\begin{example}\label{exam: mod 2}
For $n = 2$ and $3$,
all the representatives of equivalence classes of primitive $C$-type domains modulo 2 are given by the following $\Phi_0^n$.
$$
	\Phi_0^n := \left\{ \pm \displaystyle\sum_{k=1}^n i_k {\mathbf e}_k : i_k = 0 \text{ or } 1 \right\} \setminus \{ 0 \}.
$$
\end{example}

Using the decomposition $\Phi_S = \bigcup_{p \in P_S} \Psi_p$ provided by the Delone triangulation,
it can be proved that primitive $C$-type domain ${\mathcal D}_2(\Phi)$ is a union of finite $L$-type domains. As a result, the tessellation by $C$-type domains is coarser than that given by $L$-type \cite{Voronoi08}.

In what follows, $\ZZ^n$ is always regarded as a $\ZZ$-submodule in $\QQ^n$. The subspace generated by $u_1, \ldots, u_s \in \ZZ^n$ over $\QQ$ will be denoted by 
$\langle u_1, \ldots, u_s \rangle_\QQ$. The following properties of $C$-type will be used to obtain a system of representative of all the primitive $C$-type domains modulo $r$.

%\begin{lemma}
%If $C \subset \QQ^n$ is convex, symmetric about the origin, and 
%contains generators of a lattice $L$, $C$ also contains a basis of $L$. 
%\end{lemma}
%\begin{proof}
%	Supopse that $u_1, \ldots, u_s \subset C \subset \QQ^n$ are generators of $L$ with the minimal cardinality $s$. Clearly we always have $s < \infty$.
%	If $u_1, \ldots, u_s$ are linearly independent over $\QQ$, they are a basis of $L$.
%	Otherwise, take linearly dependent $u_{i_1}, \ldots, u_{i_k}$ with the minimal cardinality $2 \le k \le s$.
%We may assume the following by a coordinate transformation by $GL_n(\RR)$:
%\begin{eqnarray*}
%	u_{i_j} &=& \sum^{j}_{l=1} c_{jl} {\mathbf e_l} \quad (1 \le j \le k-1), \\
%	u_{i_k} &=& \sum^{k-1}_{l=1} d_{l} {\mathbf e_l}.
%\end{eqnarray*}
%In this case we can fined linearly There exists a linear sum of $u_{i_{k-1}}$ and $u_{i_k}$ over $\ZZ$ contained in 
%the convex hull of 
%\end{proof}

%As shown in the following proposition, each facet of a primitive $C$-type domain modulo 2
%can be associated with a set of four vectors satisfying the parallelogram law:
%$$
%	\abs{ u }^2 + \abs{ v }^2 = 2 (\abs{ (u+v)/2 }^2 + \abs{ (u-v)/2 }^2 ),
%$$
%which parallels the fact that each edge of Conway's topograph corresponds to such four vectors \cite{Conway97}.

\begin{proposition}\label{prop:facet and parallelogram law}
Suppose that two primitive $C$-type domains 
${\mathcal D}_r(\Phi_{S_1, r}) \neq {\mathcal D}_r(\Phi_{S_2, r})$ defined for $S_1, S_2 \in {\mathcal S}_{\succ 0}^n$
have an $(n(n+1)/2 - 1)$-dimensional cone as their intersection. 
In this case, if we choose vectors $u, v \in \mathbb{Z}^n$ so that 
$u \in \Phi_{S_1, r} \setminus \Phi_{S_2, r}$,
$v \in \Phi_{S_2, r} \setminus \Phi_{S_1, r}$ and $u + r \ZZ^n = v + r \ZZ^n$, then,
${\mathcal D}_r(\Phi_{S_1, r} \cup \Phi_{S_2, r}) \subset H(u, v)$
provides the common facet of ${\mathcal D}_r(\Phi_{S_1, r})$ and ${\mathcal D}_r(\Phi_{S_2, r})$.
Furthermore, the following hold:
\begin{enumerate}[(1)]
\item 
All of the following vectors belong to $\Phi_{S_1, r} \cup \Phi_{S_2, r}$.
\begin{eqnarray}\label{eq: set of vectors}
	\pm c_1 \cdot \frac{ u - v }{ r }, \quad \pm \frac{ c_2 u + (r - c_2) v }{ r } \quad (c_1 = 1, \ldots, \lfloor r/2 \rfloor,\ c_2 = 1, \ldots, r-1). 
\end{eqnarray}
For such $u$ and $v$, $(u - v)/r$ is a primitive vector.

\item\label{item: primitive case}
If $(u + v) S \tr{(u + v)}$ attains the minimal value among all the pairs of $u$ and $v$ as above, then, $u+v$ (\textit{resp.} $(u + v)/2$) is a primitive vector if $r$ is odd (\textit{resp.} even).
In this case, all the vectors in (\ref{eq: set of vectors}) also belong to $\Phi_{S_1, r} \cap \Phi_{S_2, r}$.

%If $u + v = m w$ ($0 < m \in \ZZ$) for a primitive vector $w \in \ZZ^n$, then,
%we also have 
%$u_2 := (k w + u - v)/2 \in \Phi_{S_1, r} \setminus \Phi_{S_2, r}$ and $v_2 := (k w - u + v)/2 \in \Phi_{S_2, r} \setminus \Phi_{S_1, r}$
%for any $0 < k < m$ with $k \equiv m$ mod 2.

\item If $r = 2$,
$\Phi_{S_1} \setminus \Phi_{S_2} = \{ \pm u \}$,
$\Phi_{S_2} \setminus \Phi_{S_1} = \{ \pm v \}$, and
$$\Phi_{S_1} \cap \Phi_{S_2} \cap \langle u, v \rangle_\QQ = \{ \pm (u \pm v)/2 \}.$$

\item If $r = 3$,
$\Phi_{S_1, 3} \setminus \Phi_{S_2, 3} = \{ \pm u \}$,
$\Phi_{S_2, 3} \setminus \Phi_{S_1, 3} = \{ \pm v \}$, and
$$\Phi_{S_1, 3} \cap \Phi_{S_2, 3} \cap \langle u, v \rangle_\QQ = \{ \pm (u - v)/3, \pm (2 u +v)/3, \pm (u + 2v)/3 \}.$$

\item 
There exist finitely many $C$-type domains up to the action of $GL_n(\ZZ)$.

\end{enumerate}

\end{proposition}

\begin{proof}
Since ${\mathcal D}_r(\Phi_{S_1, r}) \neq {\mathcal D}_r(\Phi_{S_2, r})$ are both primitive, there always exist $u, v$ such that $u + r\mathbb{Z}^n = v + r\mathbb{Z}^n$,
$u \in \Phi_{S_1, r} \setminus \Phi_{S_2, r}$ and $v \in \Phi_{S_2, r} \setminus \Phi_{S_1, r}$. 
It is straightforward to verify that the intersection 
${\mathcal D}_r(\Phi_{S_1, r}) \cap {\mathcal D}_r(\Phi_{S_2, r})$
coincides with 
${\mathcal D}_r(\Phi_{S_1, r} \cup \Phi_{S_2, r})$. We shall prove each item.

\begin{enumerate}[(1)]
	\item
For any $S \in {\mathcal D}_r(\Phi_{S_1, r}) \cap {\mathcal D}_r(\Phi_{S_2, r})$ and $l \in \ZZ^n$,
we have $u S \tr{l} \leq (r/2) l S \tr{l}$ and 
$v S \tr{l} \leq (r/2) l S \tr{l}$.
Hence $v S \tr{l} \leq (r/2) l S \tr{l}$ holds for any vector $w$ in Eq.(\ref{eq: set of vectors}),
which implies that they belong to $\Phi_{S_1, r} \cup \Phi_{S_2, r}$.
In particular, $(u - v)/r$ is a primitive vector, because otherwise, 
$w := \lfloor r/2 \rfloor (u - v)/r$ cannot be the shortest in $w + r \ZZ^n$.

	\item Let $w \in \ZZ^n$ be a primitive vector such that $u + v = m w$ for some integer $m > 0$.
Replacing $u, v$ with $g u, g v$ ($g \in GL_n(\ZZ)$),
we can assume that 
$u - v = r {\mathbf e}_1$ and $u + v = c {\mathbf e}_1 + d {\mathbf e}_2$
for some $c, d \in \ZZ^n$ such that $c + r$ and $d$ are even.
From this, $m$ is even if and only if $r$ is.
Fix an integer $0 < k < m$ with $k \equiv r$ mod 2. 
If we put $u_2 = (k w + u-v)/2$ and $v_2 = (k w - u + v)/2$, then it can be verified that
$u_2$ and $v_2$ belong to $\ZZ^n$ and 
they also satisfy $u_2 - v_2 = u - v \in r \ZZ^n$ and $u_2 + v_2 = (k/m)(u + v)$.
Therefore, 
\begin{eqnarray*}
(u_2 - v_2) S \tr{(u_2 + v_2)} = 0 	\text{ for any } S \in {\mathcal D}_r(\Phi_{S_1, r}) \cap {\mathcal D}_r(\Phi_{S_2, r}).
\end{eqnarray*}
Furthermore, the equalities
\begin{eqnarray*}\label{eq: u_2 and v_2}
	u_2 = \frac{k}{m} u + \frac{m-k}{m} \frac{u - v}{2},
\quad 
	v_2 = \frac{k}{m} v - \frac{m-k}{m} \frac{u - v}{2}.
\end{eqnarray*}
imply that 
$$
	u_2 S_1 \tr{l} \le (r/2) l S_i \tr{l}, \quad
	v_2 S_2 \tr{l} \le (r/2) l S_i \tr{l}
	\text{ for any } l \in \ZZ^n.
$$
Thus, we have $u_2 \in \Phi_{S_1, r} \setminus \Phi_{S_2, r}$ and 
$v_2 \in \Phi_{S_2, r} \setminus \Phi_{S_1, r}$.
From the minimalty assumption on $u$ and $v$, 
we mush have $m = 1$ if $r$ is odd, and $m = 2$ if $r$ is even.

Next, suppose that there exists another $(w, w_2) \neq (u, v)$
such that $w + r \ZZ^n = w_2 + r \ZZ^n$,
$w \in \Phi_{S_1, r} \setminus \Phi_{S_2, r}$, and $w_2 \in \Phi_{S_2, r} \setminus \Phi_{S_1, r}$.
We then have
\begin{eqnarray}\label{eq:equivalence condition from H(u, v)}
	(w - w_2) S \tr{(w + w_2)} = 0
	\text{ for any } S \in {\mathcal D}_r(\Phi_{S_1, r}) \cap {\mathcal D}_r(\Phi_{S_2, r}).
\end{eqnarray}

In each case, $w$ and $w_2$ must be as follows:
\begin{itemize}
	\item $r$ is even: $(u - v)/r$ and $(u + v)/2$ are primitive.
Since $(w - w_2)/r$ is also primitive, the following holds for some $0 \ne k \in \ZZ$.
\begin{eqnarray}\label{eq: replacing rule for even r}
	\{ w, w_2 \} = 
	\left\{ k \cdot \frac{u+v}{2} \pm \frac{u-v}{2} \right\} \text{ or }
	\left\{ k \cdot \frac{u-v}{r} \pm \frac{r}{2} \cdot \frac{u+v}{2} \right\}.
\end{eqnarray}
	\item $r$ is odd: $(u - v)/r$ and $u + v$ are primitive.
Similarly the following holds for some odd integer $k \in \ZZ$.
\begin{eqnarray}\label{eq: replacing rule for odd r}
	\{ w, w_2 \} = 
	\left\{ k \cdot \frac{u+v}{2} \pm \frac{u-v}{2} \right\} \text{ or }
	\left\{ \frac{k}{2} \cdot \frac{u-v}{r} \pm \frac{r}{2} \cdot (u + v) \right\}.
\end{eqnarray}

\end{itemize}
Such $w$ and $w_2$
can be equal to some vector in Eq.(\ref{eq: set of vectors}) only when $\{ w, w_2 \} \subset \{ \pm u, \pm v \}$, 
which implies that all the vectors in Eq.(\ref{eq: set of vectors}) belong to $\Phi_{S_1, r} \cap \Phi_{S_2, r}$.

	\item From (\ref{item: primitive case}), there exist $u, v \in \ZZ^n$ such that 
$\pm u \in \Phi_{S_1} \setminus \Phi_{S_2}$, 
$\pm v \in \Phi_{S_2} \setminus \Phi_{S_1}$,
$u + 2 \ZZ^n = v + 2 \ZZ^n$,
and both $(u - v)/2$ and $(u + v)/2$ belong to $\Phi_{S_1} \cap \Phi_{S_2}$.
From $\abs{ \Phi_{S_i} \cap \langle u, v \rangle_\QQ } = 7$ ($i = 1, 2$), the statement follows.

	\item From (\ref{item: primitive case}), there exist $u, v \in \ZZ^n$ such that 
$\pm u \in \Phi_{S_1, 3} \setminus \Phi_{S_2, 3}$, 
$\pm v \in \Phi_{S_2, 3} \setminus \Phi_{S_1, 3}$,
$u + 3 \ZZ^n = v + 3 \ZZ^n$,
and both $(u - v)/3$ and $u + v$ are primitive.
From $\abs{ \Phi_{S_i, 3} \cap \langle u, v \rangle_\QQ } = 9$ ($i = 1, 2$), the statement follows.

	\item The case $r = 2$ is already known as the Ry{\u s}hkov $C$-type, so we assume $r > 2$. 
From conditions (1) and (2), only finitely many pairs $u, v \in \Phi_{S, r}$ can correspond to a facet of ${\mathcal D}_r(\Phi_{S, r})$. 
Therefore, each primitive $C$-type domain is a polyhedral cone, and any non-empty $C$-type domain corresponds to a face of one of these cones.
Thus, it suffices to show that any primitive ${\mathcal D}_2(\Phi_{S, 2})$ intersects only finitely many primitive $C$-type domains modulo $r$. 
For any $S_2 \in {\mathcal D}_2(\Phi_{S, 2})$, we have 
$\{ 0 \} \cup \Phi_{S, 2} = {\mathbb Z}^n \cap 2 {\mathrm DV}(S_2)$. 
In such a case, there are only finitely many candidates for $\Phi_{S_2, r} = (\ZZ^n \cap r {\mathrm DV}(S_2) ) \setminus \{ 0 \}$. This implies that ${\mathcal D}_2(\Phi_{S, 2})$ can intersect only finitely many primitive $C$-type domains modulo $r$, which completes the proof.
%
%By using action of $GL_n(\ZZ)$, it is possible to 
%
%This follows from the fact that Minkowski-reduced $S$ with the first successive minimum 
%$\geq$ and $\det S = 1$ form a compact domain.
\end{enumerate}

\
\end{proof}

%In \cite{Sikiric2022}, it was proved that 
%$C$-type domains modulo $2$ form $GL_n(\ZZ)$-admissible
%decomposition of ${\mathcal S}^n_{\succ 0}$;
%$\bigcup_\alpha \sigma_\alpha$ is a
%\textit{$GL_n(\ZZ)$-admissible decomposition} if all the following hold:
%\begin{enumerate}
%	\item[(A1)] Each face of a $\sigma_\alpha$ is a $\sigma_\beta$.
%	\item[(A2)] $\sigma_\alpha \cap \sigma_\beta$ is a common face of $\sigma_\alpha$ and $\sigma_\beta$ .
%	\item[(A3)] $\gamma \sigma_\alpha$ is a $\sigma \beta$ for all $\gamma \in GLn(\ZZ)$.
%	\item[(A4)] modulo $GL_n(\ZZ)$ there are only a finite number of $\sigma_\beta$.
%	\item[(A5)] ${\mathcal S}^n_{rat, \geq 0} = \bigcup_\alpha (\sigma_\alpha \cap )
%\end{enumerate}

The decomposition (\ref{eq:decomposition of S_{> 0}^n})
is a facet-to-facet tessellation, as all the faces of primitive $C$-type domains also correspond to $\Phi_{S, r}$ of some $S \in {\mathcal S}_{\succ 0}^n$. 
Hence, (\ref{eq:decomposition of S_{> 0}^n}) is also a face-to-face tessellation 
by a theorem of Gruber and Ry{\u s}hkov \cite{Gruber89}. 

The following theorem can be verified using the algorithm in Table~\ref{table: system of representatives algorithm}.

\begin{theorem} 
\begin{enumerate}[(1)]
	\item For $n = 2$, the following $\Phi_{n, r}$ provide a system of representatives of equivalence classes of primitive $C$-type domains modulo $r = 3$ and 4.
\begin{eqnarray*}
	\Phi_{2, 3} &:=& \{ (i_1, i_2) : i_1, i_2 = -1, 0, 1 \}, \\
	\Phi_{2, 4} &:=& \Phi_{2, 3} \cup \{ \pm (2, 0), \pm (0, 2), \pm (2, 1), \pm (1, 2), \pm (2, 2) \}.
\end{eqnarray*}

	\item For $n = 3$, the following $\Phi_i$ ($i = 1$--4) provide a system of representatives of equivalence classes of primitive $C$-type domains modulo 3.
\begin{eqnarray*}
	\Phi_1 &:=& \{ (i_1, i_2, i_3) : i_1, i_2, i_3 = -1, 0, 1 \}, \\
	\Phi_2 &:=& \Phi_1 \cup \{ \pm (1, 2, 1) \} \setminus \{ \pm (1,-1, 1) \}, \\
	\Phi_3 &:=& \Phi_2 \cup \{ \pm (1, 1, 2) \} \setminus \{ \pm (1, 1,-1) \}, \\
	\Phi_4 &:=& \Phi_3 \cup \{ \pm (2, 1, 1) \} \setminus \{ \pm (-1, 1, 1) \}.
\end{eqnarray*}

	\item Let ${\mathcal D}_{min} \subset {\mathcal S}^3_{\succ 0}$ be the subset consisting 
	of all the Minkowski reduced forms defined by Eq.(\ref{eq: inequalities for being Minkowski reduced}).
	${\mathcal D}_{min}$ is contained in $\bigcup_{j=1}^{10} {\mathcal D}_3(\Phi_j)$,
	if we define 
\begin{eqnarray*}
	\Phi_5 &:=& \Phi_1 \cup \{ \pm ( 2, 1, 1) \} \setminus  \{ \pm (-1, 1, 1) \}, \\
	\Phi_6 &:=& \Phi_1 \cup \{ \pm ( 1, 1, 2) \} \setminus  \{ \pm ( 1, 1,-1) \}, \\
	\Phi_7 &:=& \Phi_1 \cup \{ \pm (-2,-1, 1) \} \setminus \{ \pm ( 1,-1, 1) \}, \\
	\Phi_8 &:=& \Phi_1 \cup \{ \pm ( 1,-1,-2) \} \setminus \{ \pm ( 1,-1, 1) \}, \\
	\Phi_9 &:=& \Phi_1 \cup \{ \pm ( 2, 1, 1), \pm ( 1, 2, 1) \} \setminus  \{ \pm (-1, 1, 1), \pm ( 1,-1, 1) \}, \\
	\Phi_{10} &:=& \Phi_1 \cup \{ \pm ( 2, 1, 1), \pm ( 1, 1, 2) \} \setminus  \{ \pm (-1, 1, 1), \pm ( 1, 1,-1) \}.
\end{eqnarray*}
\end{enumerate}
\end{theorem}

\begin{proof}
From Proposition~\ref{prop:facet and parallelogram law}, all the representatives in (1) and (2) can verified by running the algorithm in Table~\ref{table: system of representatives algorithm}.

\begin{enumerate}[(1)]
	\item[(3)] Let ${\mathcal D}_{min}$ be the set of all of $S \in {\mathcal S}^3_{\succ 0}$ that satisfies Eq.(\ref{eq: inequalities for being Minkowski reduced}).
	For each of $\Phi_i$ ($i = $1--10), We can verifiy (a) and (b) by direct calculation:
\begin{enumerate}[(a)]
	\item ${\mathcal D}_3(\Phi_i) \cap {\mathcal D}_{min}$ is not contained in a hyperplane.

	\item Any primitive domain ${\mathcal D}_3(\Phi)$ adjacent to 
	${\mathcal D}_3(\Phi_i)$, is equal to either of ${\mathcal D}_3(\Phi_j)$ ($j = $1--10) or 
${\mathcal D}_3(\Phi) \cap {\mathcal D}_{min}$ is contained in a hyperplane.
\end{enumerate}
	As a result, 
$\bigcup_{j=1}^{10} {\mathcal D}_3(\Phi_j)$ provides a decomposition of ${\mathcal D}_{min}$,
and $\iota_m$ is provided as the sorted image on ${\mathcal D}_{min}$.
\end{enumerate}
\
\end{proof}

Theorem~1 is proved as a corollary of (3).
In fact, $\iota_m(S)$ coincides with
the sorted image of the Voronoi map ${\rm vo}_{S, 3}$ modulo 3 for any $S \in {\mathcal D}_{min}$.
Theorem~2 is verified by running the algorithm of Table~\ref{table: check injectivity}.

\begin{example}\label{exam: case of embedding into 7, 8}
This example shows that the equivalence class of \( S \in \mathcal{S}^3_{\succ 0} \) cannot be determined by its sorted vonorm values and its determinant.
\begin{eqnarray*}
S_1 &:=&
\begin{pmatrix}
  4 & -1 & -1 \\
 -1 &  6 & -1 \\
 -1 & -1 &  6 \\
\end{pmatrix}
+ t
\begin{pmatrix}
   2 & -1 & -1 \\
 -1 &   2 &  0 \\
 -1 &   0 &  2 \\
\end{pmatrix}
+ t^{-1}
\begin{pmatrix}
  0  & 0 &   0 \\
  0 &  4 & -2 \\
  0 & -2 &  4 \\
\end{pmatrix}, \\
S_2 &:=&
\begin{pmatrix}
  4 & -2 & -1 \\
 -2 &  6 & -2 \\
 -1 & -2 &  8 \\
\end{pmatrix}
+ t
\begin{pmatrix}
  2 & -1 &   0 \\
 -1 &  2 & -1 \\
   0 & -1 &  2 \\
\end{pmatrix}
+ t^{-1}
\begin{pmatrix}
  0  & 0 &   0 \\
  0 &  4 & -2 \\
  0 & -2 &  4 \\
\end{pmatrix}
\end{eqnarray*}

In fact these families provide Selling-reduced forms with the identical vonorm values and determinant for any $t \ne 0$.
\begin{eqnarray*}
	& & \hspace{-5mm}
	\left(
\begin{matrix}
	\tr{{\mathbf e}}_1 S_1 {\mathbf e}_1,\
	\tr{{\mathbf e}}_2 S_1 {\mathbf e}_2,\
	\tr{{\mathbf e}}_3 S_1 {\mathbf e}_3,\ 
	\tr{{\mathbf e}}_4 S_1 {\mathbf e}_4,
\\
	\tr{({\mathbf e}_1+{\mathbf e}_2)} S_1 ({\mathbf e}_1+{\mathbf e}_2),\
	\tr{({\mathbf e}_1+{\mathbf e}_3)} S_1 ({\mathbf e}_1+{\mathbf e}_3),\
	\tr{({\mathbf e}_2+{\mathbf e}_3)} S_1 ({\mathbf e}_2+{\mathbf e}_3)
\end{matrix}
	\right)
	\\
	&=&
	\left(
\begin{matrix}
	\tr{{\mathbf e}}_1 S_2 {\mathbf e}_1,\
	\tr{{\mathbf e}}_2 S_2 {\mathbf e}_2,\
	\tr{({\mathbf e}_1+{\mathbf e}_2)} S_2 ({\mathbf e}_1+{\mathbf e}_2),\
	\tr{({\mathbf e}_2+{\mathbf e}_3)} S_2 ({\mathbf e}_2+{\mathbf e}_3)
\\
	\tr{{\mathbf e}}_3 S_2 {\mathbf e}_3,\
	\tr{{\mathbf e}}_4 S_2 {\mathbf e}_4,\
	\tr{({\mathbf e}_1+{\mathbf e}_3)} S_2 ({\mathbf e}_1+{\mathbf e}_3)
\end{matrix}
	\right)
	\\
	&=& 
	\left(
\begin{matrix}
4 + 2t,\ 4/t + 6 + 2 t,\ 4/t + 6 + 2 t,\ 4/t + 10 + 2t, \\
4/t + 8 + 2 t,\ 4/t + 8 + 2 t,\ 4/t + 10 + 4 t
\end{matrix}
	\right).
\end{eqnarray*}
\begin{eqnarray*}
	\det S_1 = \det S_2 = 48/t^2 + 188/t + 254 + 154 t + 42 t^2 + 4 t^3.
\end{eqnarray*}
It is more straightforward to construct parametrized families with identical conorm values and determinant.
\end{example}

\begin{example}\label{exam: case of embedding into 12}
This example shows that ${\mathcal LS}_3$ can be embedded into $\RR^{12}$ by a 
composition of the vonorm map $S \mapsto {\rm vo}_S$ and a polynomial map. In fact,
$V := GL_2(\ZZ/2\ZZ) \backslash \RR^{\Lambda_3}$
can be embedded into $\RR^{M}$
if there exists a surjective ring homomorphism $\pi : \RR[x_1, \ldots, x_M] \rightarrow \RR[ x_{u} : u \in \Lambda_3]^{GL_2(\ZZ/2\ZZ)}$.
As $\pi(x_i)$ is an invariant polynomial of $x_{u}$ for the action of $G$, 
$(\pi(x_1), \ldots, \pi(x_n))$ gives an embedding $V \hookrightarrow \RR^{n}$.

The number of generators of the invariant ring can also be computed using Magma.
Let $k$ be any field containing $\QQ$. The following are the primary invariants of $k[ x_{u} : u \in \Lambda_3]^{GL_2(\ZZ/2\ZZ)}$, which are algebraically independent over $k$:
\begin{eqnarray*}
	\theta_1 := \frac{1}{24} \sum_{g \in G} g \cdot x_{100}, &  &
	\theta_2 := \frac{1}{24} \sum_{g \in G} g \cdot (x_{100}^2), \\
	\theta_{31} := \frac{1}{24} \sum_{g \in G} g \cdot (x_{100}^3), &  &
	\theta_{32} := \frac{1}{24} \sum_{g \in G} g \cdot (x_{100} x_{010} x_{110}), \\
	\theta_{41} := \frac{1}{24} \sum_{g \in G} g \cdot (x_{100}^4), &  &
	\theta_{42} := \frac{1}{8} \sum_{g \in G} g \cdot (x_{100}^2 x_{010} x_{110}), \\
	& & \theta_{7} := \frac{1}{24} \sum_{g \in G} g \cdot (x_{100}^7).
\end{eqnarray*}
These and the following form a set of generators of $k[ x_{u} : u \in \Lambda_3]^{GL_2(\ZZ/2\ZZ)}$:
\begin{eqnarray*}
	\eta_{1} := \frac{1}{2} \sum_{g \in G} g \cdot  (x_{100}^2 x_{010}^2 x_{001}), & &
	\eta_{2} := \frac{1}{6} \sum_{g \in G} g \cdot  (x_{100}^2 x_{010} x_{001} x_{011}), \\
	\eta_{3} := \frac{1}{2} \sum_{g \in G} g \cdot  (x_{100}^2 x_{010}^2 x_{110} x_{001}), & &
	\eta_{4} := \frac{1}{2} \sum_{g \in G} g \cdot  (x_{100}^2 x_{010}^2 x_{001} x_{101}), \\
	& & \eta_{5} := \frac{1}{6} \sum_{g \in G} g \cdot (x_{100}^2 x_{010}^2 x_{110}^2 x_{001}).
\end{eqnarray*}
In this case, the invariant ring is a finitely generated free $P := k[\theta_1, \theta_2, \theta_{31}, \theta_{32}, \theta_{41}, \theta_{42}, \theta_{7}]$-module (\IE Cohen-Macaulay type). 
\begin{eqnarray*}
    & & 
	k[ x_{u} : u \in \Lambda_3]^{GL_2(\ZZ/2\ZZ)} \\
	&=& P 
    \oplus \left( \bigoplus_{i=1}^5 P \eta_{i} \right)
    \oplus \left( \bigoplus_{{(i, j)=(1, 2), (2, 3), (1, 4),}\atop{(1, 5), (2, 5)}}^5 P \eta_{i} \eta_{j} \right)
	\oplus P \eta_2 \eta_3 \eta_4.
\end{eqnarray*}

$k[ x_{u} : u \in \Lambda_3]^{GL_2(\ZZ/2\ZZ)}$ can be regarded as a graded ring  
$
\bigoplus_{n=0}^\infty R_n.
$  
by decomposing it by degree. Its Hilbert series is given by  
\begin{eqnarray*}
		H(t) 
&:=& \sum_{n=0}^\infty (\dim_k R_n) t^n \\
&=& \frac{ 1 + 2 t^5 + 2 t^6 + t^7 + t^{10} + 2 t^{11} + 2 t^{12} + t^{17} }{(1 - t) (1 - t^2) (1 - t^3)^2 (1 - t^4)^2 (1 - t^7)}.
\end{eqnarray*}
From this, one invariant each of degree 1 and degree 2, along with two invariants each of degrees 3, 4, 5, 6, and 7, are required to generate the ring. Therefore, $V$ cannot be embedded in $\RR^m$ with  $m < 12$, using a polynomial map.
\end{example}

\begin{table}
\caption{Algorithm for obtaining a system of representatives of all the primitive $C$-type domains modulo $r$}
\label{table: system of representatives algorithm}
\begin{minipage}{\textwidth}
\begin{tabular}{lp{95mm}}
\hline \\
{\bf Input:} & rank $n$, divisor $r$, \\
{\bf Output:} & a list \verb|OutputDomain| of $\Phi \subset \ZZ^n$ such that ${\mathcal D}_{r}(\Phi)$ provide the system of the representatives. 
\end{tabular}
\begin{tabular}{lp{110mm}}
\multicolumn{2}{l}{{\bf Algorithm}:} \\ 
1: & Set up the lists
\verb|NewD := [|$\Phi_0^n$\verb|]|,
\verb|OutputD := []|,
taking the initial $\Phi_0^n$ as described in Example~\ref{exam: S in general position}.
\\
2: & If \verb|NewD| is empty, terminate the calculation. Otherwise, proceed to 3. \\
3: & Pop an entry $\Phi_0$ from \verb|NewD|, and append it to \verb|OutputD|. 
Calculate the defining inequalities of ${\mathcal D}_r(\Phi_0)$
as follows; set up the list \verb|Ceq := []|. 
For any $0 \ne u, w \in \Phi_0$ such that $u$ and $w$ are linearly independent over $\QQ$,
put $v := u + r w$, and 
do the following (a)--(c) only if $r$ is even and $(u + v)/2$ is primitive, or $r$ is odd and $u + v$ is primitive.
Othewise, proceed to the next $u$ and $w$; \\
& (a) set $\Phi := \{ k w : 0 \le k \le r/2 \} \cup \{ u + k w : 1 \le k \le r - 1 \}$.
${\mathcal D}_r(\Phi_0)$ is contained in the half-space $H^{\leq 0}(u, v)$, 
If $\Phi \subset \Phi_0$, go to (b). Othewise, proceed to the next $u$ and $w$. \\
 & (b) If $n = 2$ or $3$, 
put $\Phi := \Phi_0 \setminus \{ \pm u \} \cup \{ \pm v \}$.
Otherwise, construct $\Phi$ by replacing each vector in $\Phi_0$ that can be written in the form on the left-hand side of an arrow, for some integer $k > 0$, with the right-hand vector, which follows from Eqs.(\ref{eq: replacing rule for even r}), (\ref{eq: replacing rule for odd r}). \\
& If $r$ is even, \\
&	\hspace{20mm}$\pm \left( k \cdot \frac{u+v}{2} + \frac{u-v}{2} \right)
	\mapsto \pm \left( k \cdot \frac{u+v}{2} - \frac{u-v}{2} \right)$, \\
&	\hspace{20mm}$\pm \left( k \cdot \frac{u-v}{r} + \frac{r}{2} \cdot \frac{u+v}{2} \right)
	\mapsto \pm \left( k \cdot \frac{u-v}{r} - \frac{r}{2} \cdot \frac{u+v}{2} \right)$.
\\
& If $r$ is odd, \\
&	\hspace{20mm}$\pm \left( k \cdot \frac{u+v}{2} + \frac{u-v}{2} \right)
	\mapsto \pm \left( k \cdot \frac{u+v}{2} - \frac{u-v}{2} \right)$, \\
&	\hspace{20mm}$\pm \left( \frac{k}{2} \cdot \frac{u-v}{r} + \frac{r}{2} \cdot (u + v) \right)
	\mapsto \pm \left( \frac{k}{2} \cdot \frac{u-v}{r} - \frac{r}{2} \cdot (u + v) \right)$.
\\
 & (c) Append $\langle \tr{v} v - \tr{u} u, \Phi \rangle$ to \verb|Ceq|. \\
4: & From the inequalities in \verb|Ceq|, calculate the inequalities that define the facets of ${\mathcal D}_r(\Phi_0)$, which can be can be obtained exactly, using techniques such as integer programming. For each entry $\langle c, \Phi \rangle$ in \verb|Ceq| that provides a facet, 
	check whether there exists $\Phi_2 \in \verb|newD| \cup \verb|OutputD|$ such that $\Phi_2 = g \Phi$ for some $g \in GL_n(\ZZ)$. If no such $\Phi_2$ exists, append $\Phi$ to \verb|newD|.
Return to 2. \\
\hline
\end{tabular}
%\footnote{$\approx$ is determined, it is possible that $G$ that $\det g = \pm 1$ is $1$ may not exist at all. In such a case, the implementation may return False, or return None as $g$ to indicate the occurrence of an exceptional situation.}.
\end{minipage}
\end{table}

\begin{table}
\caption{Algorithm for checking whether $\iota_*(C_1) \cap \iota_*(C_2) = \emptyset$ for any polyhedral cones $C_i$ contained in 
a primitive ${\mathcal D}_r(\Phi_i)$ ($i = 1, 2$). Here, $\iota_*$ stands for $\iota_s$ if $r = 2$, and $\iota_m$ if $r = 3$.
}
\label{table: check injectivity}
\begin{minipage}{\textwidth}
\begin{tabular}{lp{110mm}}
\hline \\
\multicolumn{2}{l}{{\bf Input}:} \\ 
$r$ & : 2 or 3. \\
$\Phi_i$ & : subset of $\ZZ^3$ that provides a primitive ${\mathcal D}_r(\Phi)$ ($i = 1, 2$). \\
$L_i$ & : list of defining inequalities of $C_i \subset {\mathcal D}_r(\Phi_i)$ ($i = 1, 2$). \\
\multicolumn{2}{l}{{\bf Output}:} \\ 
& False if there exists $S_i \in C_i$ ($i = 1, 2$) such that $\iota_*(S_1) = \iota_*(S_2)$. 
Otherwise, True. \\
\end{tabular}
\begin{tabular}{lp{110mm}}
\multicolumn{2}{l}{{\bf Algorithm}:} \\ 
1: & 
On each ${\mathcal D}_r(\Phi_i)$, $\iota_*$ is the composition of a certain linear map $S \mapsto F_i(S) = (F_{i,1}(S), \ldots, F_{i,13}(S)) \in \RR^{13}$ and 
and the map that sort entries in ascending order.
For each permutation $\sigma$ of the symmetric group $S_{13}$ of degree 13, calculate a list ${\mathcal R}_\sigma$ of all the pairs $(S_1, S_2)$ that generate extreme rays 
of the convex cone ${\mathcal C}_\sigma := \{ (S_1, S_2) \in C_1 \times C_2 : F_{1, k}(S) = F_{2, \sigma(k)}(S) \}$.
\\
2: & If ${\mathcal C}_\sigma$ does not contain any positive-definite pair $(S_1, S_2) \in {\mathcal S}^3_{\succ 0} \times {\mathcal S}^3_{\succ 0}$, then proceed to the next $\sigma$. 
This can be verified by checking whether the sum $(T_1, T_2) := \sum_{(S_1, S_2) \in {\mathcal R}_\sigma} (S_1, S_2)$ is a pair of positive-definite matrices. \\
3: & If $T_1$ and $T_2$ are positive-definite, 
search for $g \in GL_3(\ZZ)$ such that $g S_1 \tr{g} = S_2$ for any $(S_1, S_2) \in {\mathcal R}_\sigma$. 
Such $g$ can be searched by enumerating all $g \in GL_3(\ZZ)$ such that $g T_1 \tr{g} = T_2$
using the algorithm to generate isometries \cite{Plesken97}. 
Return False if such $g$ exists. 
\\
4: & Return True, once the loop over all $\sigma \in S_{13}$ has been completed. \\
\hline
\end{tabular}
%\footnote{$\approx$ is determined, it is possible that $G$ that $\det g = \pm 1$ is $1$ may not exist at all. In such a case, the implementation may return False, or return None as $g$ to indicate the occurrence of an exceptional situation.}.
\end{minipage}
\end{table}

\section{Determination of possible isometries $g$ for perturbated lattice parameters}

In what follows, we consider the situation where 
two reduced (Minkowski-reduced if $n > 2$) $T_1^{obs}$ and $T_2^{obs} \in {\mathcal S}_{\succ 0}^n$ 
are observed values of $g_1 S \tr{g}_1$ and $g_2 S \tr{g}_2$ of some unknown $S \in {\mathcal S}_{\succ 0}^n$ and $g_i \in GL_n(\ZZ)$.

In order to limit the number of potential candidates for $g := g_2 g_1^{-1}$ such that 
$
	g T_1^{obs} \tr{g} \approx T_2^{obs},
$ to a finite set, it is necessary to find a general assumption on $T_1^{obs}$ and $T_2^{obs}$ that is derived from the situation where both of them are positive-definite. The following is suggested here. 
\begin{description}
\item[{\rm ($C_{T_1^{obs}, T_2^{obs}}$)}] 
Assume that the reduced $T_1^{obs}, T_2^{obs} \in {\mathcal S}_{\succ 0}^n$ are nearly equivalent 
over $\ZZ$ in the sense that (i) and (ii) do not hold for any $i = 1, \ldots, n$:
\begin{enumerate}[(i)]
	\item ${\mathbf e}_i \not\ll_{T_1^{obs}} \tr{g}\, {\mathbf e}_i$, 
	\item ${\mathbf e}_i \not\ll_{T_2^{obs}} \tr{g}^{-1} {\mathbf e}_i$,
\end{enumerate}
\end{description}

For the definition of the symbol $\not\ll_{T_i^{obs}}$, 
recall that a binary relation $\lesssim$ defined on a set $P$ is called a \textit{preorder} if it satisfies:

\vspace{2mm}
\begin{enumerate}[i.]
	\item (Reflexive) $a \lesssim a$ for all $a \in P$, and
	\item (Transitive) if $a \lesssim b$ and $b \lesssim c$ then $a \lesssim c$ for all $a, b, c \in P$.
\end{enumerate}

\vspace{2mm}
For any $S \in {\mathcal S}^n_{\succ 0}$, 
we define a preorder $\lesssim_S$ on $\ZZ^n$ by
$$
	v_1 \lesssim_S v_2 \underset{def}{\Leftrightarrow} v_1 S \tr{v}_1 \le v_2 S \tr{v}_2.
$$
If $v_1 \lesssim_S v_2$ but not $v_2 \lesssim_S v_1$, 
we simply denote $v_1 <_S v_2$.

\begin{definition}
	Let $\lambda_1 > 0$ be the first successive minimum of $S \in {\mathcal S}^n_{\succ 0}$, \IE 
$$
	\lambda_1 = \inf_{0 \ne v \in \ZZ^n} v S \tr{v}.
$$
	For any $v_1, v_2 \in \ZZ^n$, 
we shall say that $v_1 <_S v_2$ stably holds
if 
$
	 v_2 S \tr{v}_2 - v_1 S \tr{v}_1 \ge \lambda_1
$,
which will be denoted by $v_1 \ll_S v_2$.
\end{definition}
For example, $0 \ll_S v$ holds for any $0 \ne v \in \ZZ^n$.
Consequently, the above (i), (ii) of $C_{T_1^{obs}, T_2^{obs}}$ just ensure that the following do not happen.

\vspace{2mm}
\begin{enumerate}[(i)]
\item
$\tr{{\mathbf e}}_i g T_1^{obs} \tr{g}\, {\mathbf e}_i - \tr{{\mathbf e}}_i T_1^{obs} {\mathbf e}_i \ge \tr{{\mathbf e}}_1 T_1^{obs} {\mathbf e}_1$,

	\item
$\tr{{\mathbf e}}_i g^{-1} T_2^{obs} \tr{g}^{-1} {\mathbf e}_i - \tr{{\mathbf e}}_i T_2^{obs} {\mathbf e}_i + \ge \tr{{\mathbf e}}_1 T_2^{obs} {\mathbf e}_1$,
\end{enumerate}
\vspace{2mm}These are justified by the fact that 
the diagonal entries of any Minkowski-reduced forms coincide with their successive minima for $n \le 4$,
and the successive minima of $T_1^{obs}$ and $T_2^{obs}$ should be nearly equal
if they are nearly equivalent over $\ZZ$.

For constructing an algorithm to classify lattices into their Bravais types,  
the following assumption was used to avoid iterations over nearly reduced bases in\cite{Tomiyasu2012}. 
\begin{description}
\item[ (${\mathcal A}_{n,d}$) ] If the $n$-by-$n$ metric tensor $S$ satisfies $S \bullet T := {\rm trace}(S T) \geq d \cdot v S \tr{v}$ for some $0 \ne v \in \ZZ^n$ and $T \in {\mathcal S}^n$, its observed value $S^{obs}$ also satisfies $S^{obs} \bullet T > 0$.
\end{description}
Since $T$ can be restricted to the ones in the following form in the proof of \cite{Tomiyasu2012}, the condition $C_{T_1^{\mathrm{obs}}, T_2^{\mathrm{obs}}}$ is weaker than $\mathcal{A}_{n,d}$.
$$
	T = \sum_{i=1}^{s} g {\mathbf e}_i \tr{\mathbf e}_i \tr{g} - \sum_{i=1}^{s} {\mathbf e}_i \tr{\mathbf e}_i, \quad g \in GL_n(\ZZ).
$$

In order to construct an algorithm that gives all the potential isometries,
we show several statements derived from $C_{T_1^{obs}, T_2^{obs}}$.
\begin{lemma}\label{lem: stably greater}
\begin{enumerate}
	\item $v \ll_S m v$ for any 
	$S = (s_{ij}) \in {\mathcal S}^n_{\succ 0}$ and $m \in \ZZ$ such that $m \ne 0, \pm 1$.

	\item \label{item: n=2} 
	Suppose that 
	$S = (s_{ij}) \in {\mathcal S}^2_{\succ 0}$ 
is reduced \IE satisfies 
$
	2 \abs{ s_{12} } \le s_{11} \le s_{22}.
$
If $v \in \ZZ^2$ is not a multiple of ${\mathbf e}_1$ or ${\mathbf e}_2$ and is not equal to $\pm ({\mathbf e}_1 + {\mathbf e}_2)$, 
then ${\mathbf e}_i \ll_S v$ for both of $i = 1, 2$.

	\item 
	If 
$S = (s_{ij}) \in {\mathcal S}^3_{\succ 0}$ 
is Selling-reduced, and $v \in \ZZ^3$ is not a linear sum of any pairs of ${\mathbf e}_1, {\mathbf e}_2, {\mathbf e}_3$
and ${\mathbf e}_4 := -({\mathbf e}_1 + {\mathbf e}_2 + {\mathbf e}_3)$,
then, (i) and (ii) hold:
\begin{enumerate}[(i)]
	\item ${\mathbf e}_i \ll_S v$ for any  $1 \le i \le 4$,  
	\item \label{item: e_i + e_j} ${\mathbf e}_i + {\mathbf e}_j \ll_S v$ for at least two of $(i, j) = (1, 2), (1, 3), (2, 3)$.
\end{enumerate}
In particular, 
${\mathbf e}_i + {\mathbf e}_j  \ll_S  v$ holds for any $(i, j) = (1, 2), (1, 3), (2, 3)$,
if some $( i_2, j_2 ) \in \{ ( 1, 2 ), ( 1, 3 ), ( 2, 3 ) \}$ not equal to $( i, j )$ satisfies
${\mathbf e}_i + {\mathbf e}_j \lesssim_S {\mathbf e}_{i_2} + {\mathbf e}_{j_2}$.

\end{enumerate}
\end{lemma}

\begin{proof}
\begin{enumerate}
	\item Clear.
	
	\item Since $S$ is Selling reduced, 
	all the non-diagonal entries 
	of $\tilde{S} = (\tilde{s}_{ij})$ taken as in Eq.(\ref{eq: definition of tilde{S}}) are negative.
	For these $S$ and $\tilde{S}$, the following equality holds:
	$$
		\tr{v} S v = -\tilde{s}_{12} (v_1 - v_2)^2 - \tilde{s}_{13} (v_1 - v_3)^2 - \tilde{s}_{23} (v_2 - v_3)^2
\text{ for any }
v := \sum_{i=1}^3 v_i {\mathbf e}_i \in \ZZ^2.
	$$
It can be easily confirmed that two of $v_1, v_2, v_3$ are equal if and only if
$v$ is a multiple of ${\mathbf e}_1, {\mathbf e}_2$, or $-({\mathbf e}_1 + {\mathbf e}_2)$. 
If this is not the case,
all of the $(v_i - v_j)^2$ are $\ge 1$, and at least one of them is $\ge 4$. 
Thus, for some $1 \le j < k \le 3$, 
	$$
		\tr{v} S v \ge \tilde{s}_{jj} + \tilde{s}_{kk}.
	$$
From $\tilde{s}_{11} \le \tilde{s}_{22} \le \tilde{s}_{33}$, 
${\mathbf e}_i \ll_S v$ is proved for both $i = 1, 2$.
Furthermore, for any $m \in \ZZ$ with $m \ne 0, \pm 1$,
$$
	{\mathbf e}_i \lesssim_S ({\mathbf e}_1 + {\mathbf e}_2) \ll_S m ({\mathbf e}_1 + {\mathbf e}_2),
$$
which proves the statement.

	\item For any $v := \sum_{i=1}^4 v_i {\mathbf e}_i \in \ZZ^3$,
the following equality holds:
$$
\tr{v} S v = -\sum_{1 \le i < j \le 4} \tilde{s}_{ij} (v_i - v_j)^2. 
$$
Two of $v_1, v_2, v_3, v_4$ are equal to each other 
if and only if
$v$ belongs to a $\ZZ$-module 
generated by ${\mathbf e}_i$ and ${\mathbf e}_j$ for some $1 \le i < j \le 4$.
If this is not the case, it is possible to choose $\{ i, j, k, l \} = \{ 1, 2, 3, 4 \}$ so that 
$v_i < v_j < v_k < v_l$. We then have
\begin{eqnarray*}
\tr{v} S v 
& \ge & -\tilde{s}_{ij} - \tilde{s}_{jk} - \tilde{s}_{kl} -4 \tilde{s}_{ik} - 4 \tilde{s}_{jl} - 9 \tilde{s}_{il} \\
& \ge & \tilde{s}_{ii} + \tilde{s}_{ll} + (\tilde{s}_{ii} + \tilde{s}_{jj} + 2 \tilde{s}_{ij}) 
\ge  
	(\tilde{s}_{ii} + \tilde{s}_{ll} + 2 \tilde{s}_{il})
	+ (\tilde{s}_{ii} + \tilde{s}_{jj} + 2 \tilde{s}_{ij}).
\end{eqnarray*}
Similarly, 
$$
\tr{v} S v 
\ge 
\max \{ 
	\tilde{s}_{ii} + \tilde{s}_{kk}, \ \tilde{s}_{jj} + \tilde{s}_{ll}
\}.
$$
Consequently, (i) and (ii) are proved. The last statement is clear because 
${\mathbf e}_i + {\mathbf e}_j \not\ll_S v$
but ${\mathbf e}_{i_2} + {\mathbf e}_{j_2} \ll_S v$
is impossible
in this case.

\end{enumerate}
\end{proof}

By applying the condition $C_{T_1^{obs}, T_2^{obs}}$ 
to the reduced $T_1^{obs}, T_2^{obs} \in {\mathcal S}_{\succ 0}^n$, 
we can derive some properties of $g \in GL_3(\ZZ)$ such that $g T_1^{obs} \tr{g} \approx T_2^{obs}$.

\begin{theorem}\label{thm: candidate for g}
	\begin{enumerate}
		\item 
		In case of $n = 2$, 
		$C_{T_1^{obs}, T_2^{obs}}$ implies that 
		the following holds for $i = 1, 2$: 
$$
	\tr{g}\, {\mathbf e}_i = \pm {\mathbf e}_1, \pm {\mathbf e}_2, \text{ or } \pm ({\mathbf e}_1+{\mathbf e}_2)
$$

		\item 
		In case of $n = 3$, 
		$C_{T_1^{obs}, T_2^{obs}}$ implies that 
		$\tr{g}\, {\mathbf e}_i \in \Psi_i$ ($i = 1, 2, 3$), where each $\Psi_i$ is defined by: 
\begin{eqnarray*}
	\Psi_1 = \Psi_2 &:=& \left\{ 
	\begin{matrix}
	\pm {\mathbf e}_1, \pm {\mathbf e}_2, \pm {\mathbf e}_3, 
	\pm ({\mathbf e}_1 + {\mathbf e}_2), 
	\pm ({\mathbf e}_1 + {\mathbf e}_3), 
	\pm ({\mathbf e}_2 + {\mathbf e}_3), \\
	\pm ({\mathbf e}_1 - {\mathbf e}_3), 
	\pm ({\mathbf e}_1 + {\mathbf e}_2 + {\mathbf e}_3),
	\pm ({\mathbf e}_1 - {\mathbf e}_2 - {\mathbf e}_3)
	\end{matrix}
	\right\}, \\
	\Psi_3 &:=& \Psi_1 \cup \{
	\pm ({\mathbf e}_1 - {\mathbf e}_2) \}
\end{eqnarray*}

	\end{enumerate}
\end{theorem}

The following fact presented in \cite{Balashov57}, will be used for the proof of Theorem~\ref{thm: candidate for g}:
\begin{theorem}\label{thm: Balashov-Ursell}
	Let $S \in {\mathcal S}_{\succ 0}^3$ be a Selling-reduced form  
	such that $\tilde{S} = (s_{ij})$ of Eq.(\ref{eq: definition of tilde{S}}) satisfies $s_{11} \le s_{22} \le s_{33} \le s_{44}$.
	$S$ can be transformed into a Minkowski-reduced form $\sigma_j S \tr{\sigma}_j$
	using the following $\sigma_j \in GL_3(\ZZ)$ in each case:
\begin{enumerate}[(i)]
	\item\label{item: sigma1} $s_{11} + 2 s_{12} < 0$: 	
	$
	\sigma_1 := 
	\begin{pmatrix}
		-1 & 0 & 0 \\
		1 & 1 & 0 \\
		0 & 0 & 1 \\
	\end{pmatrix}
	$.

	\item\label{item: sigma2} $s_{11} + 2 s_{13} < 0$: 	
	$
	\sigma_2 := 
	\begin{pmatrix}
		1 & 0 & 0 \\
		0 & 1 & 0 \\
		1 & 0 & 1 \\
	\end{pmatrix}
	$.

	\item\label{item: sigma3} $s_{22} + 2 s_{23} < 0$: 	
	$
	\sigma_3 := 
	\begin{pmatrix}
		1 & 0 & 0 \\
		0 & 1 & 0 \\
		0 &-1 &-1 \\
	\end{pmatrix}
	$.

	\item\label{item: sigma4} Otherwise, $\sigma_4 = I$ (\IE $S$ is Minkowski-reduced).
\end{enumerate}
\end{theorem}
Note that the inequalities of (i)--(iv) do not hold simultaneously. 
For example, 
the sum of $s_{11} + 2 s_{13}$ and $s_{22} + 2 s_{23}$ cannot be both negative, due to the following:
$$
	s_{33} \le s_{44} = s_{11} + s_{22} + s_{33} + 2 s_{12} + 2 s_{13} + 2 s_{23}.
$$

\begin{proof}[Proof of Theorem~\ref{thm: candidate for g}] \

\begin{enumerate}
	\item $S$ is also Selling-reduced in this case.
	Therefore, if ${\mathbf e}_2 \not\ll_{T_1^{obs}} v$, 
	2.~of Lemma~\ref{lem: stably greater} implies
$$
	v = m {\mathbf e}_i \text{ or } \pm ({\mathbf e}_1 + {\mathbf e}_2) \text{ for some } m \in \ZZ.
$$
	From $C_{T_1^{obs}, T_2^{obs}}$, 
	$\tr{g}\, {\mathbf e}_i$ is equal to such $v$ ($i = 1, 2$).
	From $g \in GL_2(\ZZ)$, 
	the case of $m \ne \pm 1$ is excluded.

	\item Take $\sigma_j$
		to transform $T^{obs}_1$ to its Selling-reduced form $T_1^{sel} := \sigma_j^{-1} T_1^{obs} \tr{\sigma}_j^{-1}$ as in Theorem~\ref{thm: Balashov-Ursell}.
We shall prove that 
\begin{description}
\item[{\rm (*)}] If $\tr{\sigma}_j {\mathbf e}_3 \not\ll_{T_1^{sel}} v$, then
$v \in \langle {\mathbf e}_k, {\mathbf e}_l \rangle$ for some $1 \le k < l \le 4$,
\end{description}
where $\langle {\mathbf e}_k, {\mathbf e}_l \rangle$ is the $\ZZ$-module generated by 
${\mathbf e}_k$ and ${\mathbf e}_l$.

In case of (i) and (iv), (*) immediately follows from 3.~of Lemma~\ref{lem: stably greater}, 
because $\tr{\sigma}_j {\mathbf e}_3 = {\mathbf e}_3$.
In case of (ii) or (iii),
\begin{eqnarray*}
	\{ \pm \tr{\sigma}_j^{-1} ({\mathbf e}_k + {\mathbf e}_l) : 1 \le k < l \le 3 \}
	&=& \{ 
		\pm ({\mathbf e}_1 + {\mathbf e}_2), 
		\pm {\mathbf e}_3 
		\pm (- {\mathbf e}_1 + {\mathbf e}_2 + {\mathbf e}_3) 
		\},
\end{eqnarray*}
and thus,
\begin{eqnarray*}
	\{ 
		\pm \tr{\sigma}_j ({\mathbf e}_1 + {\mathbf e}_2), 
		\pm \tr{\sigma}_j {\mathbf e}_3 
		\pm \tr{\sigma}_j (- {\mathbf e}_1 + {\mathbf e}_2 + {\mathbf e}_3) 
		\}
	&=& 
	\{ \pm ({\mathbf e}_k + {\mathbf e}_l) : 1 \le k < l \le 3 \}.
\end{eqnarray*}
From 
${\mathbf e}_3 \lesssim_{T_1^{obs}} - {\mathbf e}_1 + {\mathbf e}_2 + {\mathbf e}_3$,
we have 
$$
\tr{\sigma}_j {\mathbf e}_3 \lesssim_{T_1^{sel}} \tr{\sigma}_j (- {\mathbf e}_1 + {\mathbf e}_2 + {\mathbf e}_3).
$$
From 3 (ii) of Lemma~\ref{lem: stably greater}, 
if $v \notin \langle {\mathbf e}_k, {\mathbf e}_l \rangle$ for any $1 \le k < l \le 4$, then 
$\tr{\sigma}_j {\mathbf e}_3 \not\ll_{T_1^{sel}} v$. So (*) holds even in this case.
	
From $C_{T_1^{obs}, T_2^{obs}}$, 
$\tr{\sigma}_j {\mathbf e}_3 \ll_{T_1^{sel}} \tr{\sigma}_j \tr{g} {\mathbf e}_i$ for none of $i = 1, 2, 3$. 
From (*), the following holds for any $i = 1, 2, 3$: 
\begin{eqnarray}\label{eq: sigma_j^T g e_i in < e_k, e_l >}
	\tr{\sigma}_j \tr{g} {\mathbf e}_i \in \langle {\mathbf e}_k, {\mathbf e}_l \rangle \text{ for some $1 \le k < l \le 4$.}
\end{eqnarray}

Next, we prove that $C_{T_1^{obs}, T_2^{obs}}$ and Eq.(\ref{eq: sigma_j^T g e_i in < e_k, e_l >}) imply
\begin{eqnarray*}
(**)
\begin{cases}
	\tr{g} {\mathbf e}_1, \tr{g} {\mathbf e}_2 & \in \Phi := \{ 
	\pm \tr{\sigma}_j^{-1} {\mathbf e}_k,\
	\pm \tr{\sigma}_j^{-1} {\mathbf e}_l,\
	\pm \tr{\sigma}_j^{-1} ({\mathbf e}_l + {\mathbf e}_l) : 1 \le k < l \le 4 \}, \\
	\tr{g} {\mathbf e}_3 & \in \Phi \cup \left\{ 
	\pm ({\mathbf e}_1 - {\mathbf e}_2) \right\}.
\end{cases}
\end{eqnarray*}

From Eq.(\ref{eq: sigma_j^T g e_i in < e_k, e_l >}), there exist $1 \le k < l \le 4$ such that
$
	\tr{g} {\mathbf e}_i \in \langle \tr{\sigma}_j^{-1} {\mathbf e}_k, \tr{\sigma}_j^{-1} {\mathbf e}_l \rangle.
$ 
From $C_{T_1^{obs}, T_2^{obs}}$, we also have ${\mathbf e}_3 \not \ll_{T_1^{obs}} \tr{g} {\mathbf e}_i$.
Since 
$$
	\begin{pmatrix}
		\tr{{\mathbf e}}_k \\ \tr{{\mathbf e}}_l
	\end{pmatrix}
	T_1^{sel}
	\begin{pmatrix}
		{\mathbf e}_k & {\mathbf e}_l
	\end{pmatrix}
	=
	\begin{pmatrix}
		\tr{{\mathbf e}}_k \\ \tr{{\mathbf e}}_l
	\end{pmatrix}
	\sigma_j^{-1} T_1^{obs} \tr{\sigma}_j^{-1}
	\begin{pmatrix}
		{\mathbf e}_k & {\mathbf e}_l
	\end{pmatrix}
$$
is also Selling-reduced, all the non-diagonal entries of the following matrix are negative:
$$
	\begin{pmatrix}
		\tr{{\mathbf e}}_k \\ \tr{{\mathbf e}}_l \\ -\tr{ ({\mathbf e}_k + {\mathbf e}_l) }
	\end{pmatrix}
	\sigma_j^{-1} T_1^{obs} \tr{\sigma}_j^{-1}
	\begin{pmatrix}
		{\mathbf e}_k & {\mathbf e}_l & -({\mathbf e}_k + {\mathbf e}_l)
	\end{pmatrix}.
$$

First suppose that $(k, l) \ne (1, 2)$.
Then
$\langle \tr{\sigma}_j^{-1} {\mathbf e}_k, \tr{\sigma}_j^{-1} {\mathbf e}_l \rangle \not\subset \langle {\mathbf e}_1, {\mathbf e}_2 \rangle$ regardless of the choice of $j$. 
Therefore, 
${\mathbf e}_3 \lesssim_{T_1^{obs}} v$ holds for at least two of $v = \tr{\sigma}_j^{-1} {\mathbf e}_k, \tr{\sigma}_j^{-1} {\mathbf e}_l, \tr{\sigma}_j^{-1} ({\mathbf e}_k + {\mathbf e}_l)$.
If $\tr{g} {\mathbf e}_i \notin \Phi$, 
then \ref{item: n=2} of Lemma~\ref{lem: stably greater} implies ${\mathbf e}_3 \ll_{T_1^{obs}} \tr{g} {\mathbf e}_i$, which contradicts with $C_{T_1^{obs}, T_2^{obs}}$. Thus, 
(**) is proved for $\tr{g} {\mathbf e}_i$ in this case.

Next suppose that $(k, l) = (1, 2)$. Then  
$\tr{g} {\mathbf e}_i \in \langle {\mathbf e}_1, {\mathbf e}_2 \rangle$ holds regardless of the choice of $j$. If $i = 1, 2$, then $\tr{g} {\mathbf e}_i \in \Phi$ is proved if 
$$
	\tr{g} {\mathbf e}_i \in \{ 
	\pm {\mathbf e}_1,\
	\pm {\mathbf e}_2,\
	\pm ({\mathbf e}_1 + {\mathbf e}_2) \}.
$$
This holds clearly because otherwise,
$
	{\mathbf e}_2 \ll_{T_1^{obs}} \tr{g} {\mathbf e}_i
$ is obtained, which contradicts with $C_{T_1^{obs}, T_2^{obs}}$.
If $i = 3$, it is possible to choose $v \in \ZZ^3$ so that $\langle {\mathbf e}_3, v \rangle = 
\langle \tr{g}^{-1} {\mathbf e}_1, \tr{g}^{-1} {\mathbf e}_2 \rangle$
and 
$$
	\begin{pmatrix}
		\tr{{\mathbf e}}_3 \\ v \\ - \tr{{\mathbf e}}_3 - v
	\end{pmatrix}
	T_2^{obs}
	\begin{pmatrix}
		{\mathbf e}_3 & v & - \tr{{\mathbf e}}_3 - v
	\end{pmatrix}
$$
is Selling-reduced.
In this case,
$$
\tr{g}^{-1} {\mathbf e}_1,
\tr{g}^{-1} {\mathbf e}_2
\in 
\{ \pm {\mathbf e}_3, \pm v, \pm ({\mathbf e}_3+v) \}
$$
is required, because otherwise, 
${\mathbf e}_3 \ll_{T_2^{obs}} \tr{g}^{-1} {\mathbf e}_j$, which contradicts with $C_{T_1^{obs}, T_2^{obs}}$.
Therefore, ${\mathbf e}_3$ is equal to 
either of $\pm \tr{g}^{-1} {\mathbf e}_1$,  
$\pm \tr{g}^{-1} {\mathbf e}_2$,
$\pm \tr{g}^{-1} ({\mathbf e}_1 \pm {\mathbf e}_2)$, which proves (**).
The statement is obtained from (**) immediately.
\end{enumerate}
\end{proof}

As a natural consequence of Theorem~\ref{thm: candidate for g}, 
Table~\ref{table: algorithm} presents an algorithm that returns all the potential isometries 
\( g \in GL_3(\ZZ) \) for a given Minkowski-reduced \( S \in {\mathcal S}^3_{\succ 0} \).
In particular, it can be seen that even if another Minkowski-reduced 
\( S_2 \in {\mathcal S}^3_{\succ 0} \) is not given, and perturbations in $S$ and $S_2$ are taken into account, the number of the potential candidates for $g$ is already finite.

\begin{table}[htbp]
\caption{ Algorithm that returns all potential isometries $g \in GL_3(\ZZ)$ such that $g S \tr{g} \approx S_2$ for a given Minkowski-reduced $S \in {\mathcal S}^3_{\succ 0}$ and another Minkowski-reduced $S_2 \in {\mathcal S}^3_{\succ 0}$ not given yet. }
\label{table: algorithm}
\begin{minipage}{\textwidth}
\begin{tabular}{lp{85mm}}
\hline \\
\multicolumn{2}{l}{{\bf Input}:} \\ 
$[ \mu_{k} : k = 1, \ldots, 13 ]$ & : image of a Minkowski-reduced $S \in {\mathcal S}^3_{\succ 0}$ by $\iota_m$ ($i = 1, 2$). \\
$[ v_{k} : k = 1, \ldots, 13 ]$ & : list of integral vectors such that $\tr{v_{k}} S v_{k} = \mu_{k}$. \\
\multicolumn{2}{l}{{\bf Output}:} \\
$\Omega$: & 
list of $g \in GL_3(\ZZ)$ such that ${\mathbf e}_k \not\ll_{S} \tr{g} {\mathbf e}_k$ for any $k = 1, 2, 3$. \\
\end{tabular}
\begin{tabular}{lp{110mm}}
\multicolumn{2}{l}{{\bf Algorithm}:} \\ 
1: & Set up $\lambda_{j} := \tr{{\mathbf e}}_j S {\mathbf e}_j$ ($j = 1, 2, 3$) using the input $\mu_{1, k}$ and $v_{i, k}$ ($k=1, \ldots, 13$). \\
2: & Append to $\Omega$ all $g \in GL_3(\mathbb{Z})$ such that the (a) $j$'th row vector \( g_j \) is equal to some $v_{k}$ that satisfies 
$
\tr{v_k} S v_k < \lambda_{j} + \lambda_{1}
$ ($j = 1, 2, 3$) and (b) every row of $g^{-1}$ is equal to some vector of $\Phi_3$ in Theorem~\ref{thm: candidate for g}. \\
3: & Return $\Omega$. \\
\hline \\
\end{tabular}
\end{minipage}
\end{table}

\bibliography{sn-bibliography}% common bib file


\end{document}